\documentclass[a4paper,english,12pt]{article}

\usepackage{amsmath,amsfonts,amssymb,amsthm,bbm,graphics,epsfig,psfrag,epstopdf,dsfont,esint,mathtools,yhmath,geometry,paralist,subfigure,float,color,soul,hyperref}
\usepackage[active]{srcltx}
\usepackage[dvipsnames]{xcolor}

\setstcolor{red}
\newtheorem{theorem}{Theorem}[section]

\newtheorem{pro}[theorem]{Proposition}

\newtheorem{remark}[theorem]{Remark}

\newtheorem{open}[theorem]{Open Problem}

\newcommand{\R}{\mathbb{R}}
\def\N{\mathbb{N}}

\def\epsilon{\varepsilon}

\def\dsp{\displaystyle}

\newcommand{\be}{\begin{equation}}
\newcommand{\ee}{\end{equation}}
\newcommand{\baa}{\begin{array}}
\newcommand{\eaa}{\end{array}}
\newcommand{\ba}{\begin{eqnarray}}
\newcommand{\ea}{\end{eqnarray}}
\numberwithin{equation}{section}

\begin{document}
\date{}
\title{{\bf{Log-concavity and anti-maximum principles for semilinear and linear elliptic equations}}\vskip 0.5cm
Log-concavit\'e et principes d'anti-maximum pour des \'equations elliptiques lin\'eaires et semi-lin\'eaires}
\author{Fran\c cois Hamel and Nikolai Nadirashvili\\
\\
\footnotesize{Aix Marseille Univ, CNRS, I2M, Marseille, France \thanks{This work has been supported by the French Agence Nationale de la Recherche (ANR), in the framework of the ReaCh project (ANR-23-CE40-0023-02). The authors are grateful to V.~Bobkov, C.~De~Coster, R.~Pardo and G.~Sweers for helpful discussions on some relevant references. E-mail addresses: francois.hamel@univ-amu.fr, nikolay.nadirashvili@univ-amu.fr}}}
\maketitle

\begin{center}
{\it \`A Ha{\"{\i}}m Brezis, avec respect et admiration pour un grand ma{\^{\i}}tre de l'analyse}
\end{center}

\vskip 0.5cm
\begin{abstract}
\noindent{}This paper is concerned with existence and qualitative properties of positive solutions of semilinear elliptic equations in bounded domains with Dirichlet boundary conditions. We show the existence of positive solutions in the vicinity of the linear equation and the log-concavity of the solutions when the domain is strictly convex. We also review the standard results on the log-concavity or the more general quasi-concavity of solutions of elliptic equations. The existence and other convergence results especially rely on the maximum principle, on a quantified version of the anti-maximum principle, on the Schauder fixed point theorem, and on some a priori estimates.
\vskip 6pt
\centerline{\bf{R\'esum\'e}}
\noindent{}Cet article porte sur des questions d'existence et des propri\'et\'es qualitatives de solutions strictement positives d'\'equations elliptiques semi-lin\'eaires dans des domaines born\'es avec des conditions au bord de type Dirichlet. Nous montrons l'existence de solutions strictement positives au voisinage d'une \'equation lin\'eaire et la log-concavit\'e des solutions lorsque le domaine est strictement convexe. Nous pr\'esentons \'egalement une synth\`ese des r\'esultats classiques sur la log-concavit\'e et sur la notion plus g\'en\'erale de quasi-concavit\'e pour des \'equations elliptiques. L'existence de solutions et des r\'esultats suppl\'ementaires de convergence reposent notamment sur le principe du maximum, sur une version quantifi\'ee du principe d'anti-maximum, sur le th\'eor\`eme de point fixe de Schauder, et sur des estimations a priori.
\vskip2pt
\noindent{\small{\it{Mots-cl\'es}} : \'Equations elliptiques semi-lin\'eaires~; propri\'et\'es qualitatives~; log-concavit\'e~; principe d'anti-maximum.} 
\end{abstract}


\section{Introduction and main results}\label{intro}

This paper is concerned with existence, convergence and qualitative properties of positive solutions of semilinear elliptic equations of the type
$$\Delta u+f(x,u,\nabla u)=0$$
in bounded domains $\Omega\subset\R^N$, with Dirichlet boundary conditions on $\partial\Omega$. More precisely, we will focus on equations which are close in some sense to the linear case, and the qualitative results refer to concavity properties of the solutions inherited from the domain when it is convex. We both show some new concavity properties beyond the linear case and we will also review some known results on the quasi-concavity and log-concavity of positive functions. To establish the existence and convergence results, we will use the Schauder fixed point theorem as well as the maximum and anti-maximum principles. We especially prove new quantified versions of the anti-maximum principle, and we also review some of the main contributions on this topic. Throughout the paper, $N$ is any positive integer and~$\Omega$ is a non-empty bounded open connected subset of $\R^N$ (a domain), with boundary of class~$C^{1,1}$.

Let $\lambda_1>0$ be the smallest eigenvalue of the operator $-\Delta$ with Dirichlet boundary conditions on $\partial\Omega$. It is both given by the Rayleigh variational formula
$$\lambda_1=\min_{\phi\in H^1_0(\Omega),\,\|\phi\|_{L^2(\Omega)}=1}\ \int_\Omega|\nabla\phi|^2,$$
where $|\ |$ denotes the Euclidean norm in $\R^N$, and it is characterized by the existence of an eigenfunction $\varphi_1\in C^\infty(\Omega)\cap C^1(\overline{\Omega})$ satisfying
\be\label{varphi1}\left\{\baa{rcll}
\Delta\varphi_1+\lambda_1\varphi_1 & \!\!=\!\! & 0 & \!\hbox{in $\Omega$},\vspace{3pt}\\
\varphi_1 & \!\!>\!\! & 0 & \!\hbox{in $\Omega$},\vspace{3pt}\\
\varphi_1 & \!\!=\!\! & 0 & \!\hbox{on $\partial\Omega$}.\eaa\right.
\ee
By the Krein-Rutman theorem, $\varphi_1$ is unique up to multiplication by positive constants, and the Hopf lemma asserts that $\nu\cdot\nabla\varphi_1<0$ on $\partial\Omega$, where $\nu$ is the outward unit normal to~$\Omega$. We denote $\lambda_1<\lambda_2\le\lambda_3\le\cdots$ the sequence of eigenvalues of $-\Delta$ with Dirichlet boundary conditions on $\partial\Omega$, see e.g.~\cite{HB}.

We now consider some nonlinear perturbations of the equation~\eqref{varphi1}, of the type
\be\label{eq}\left\{\baa{rcll}
\Delta u+(\lambda_1+\epsilon)u-\delta g(x,u,\nabla u) & \!\!=\!\! & 0 & \!\hbox{in $\Omega$},\vspace{3pt}\\
u & \!\!>\!\! & 0 & \!\hbox{in $\Omega$},\vspace{3pt}\\
u & \!\!=\!\! & 0 & \!\hbox{on $\partial\Omega$},\eaa\right.
\ee
where $\epsilon$ and $\delta$ are two real parameters, and $\epsilon$ is either negative, or positive and small. The real-valued function $g:(x,s,p)\mapsto g(x,s,p)$ defined in $\overline{\Omega}\times[0,+\infty)\times\R^N$ is assumed to satisfy
\be\label{hypg}\!\!\!\left\{\baa{l}
g\hbox{ is H\"older-continuous in $\overline{\Omega}\times[0,A]\times[-A,A]^N$ for every $A>0$},\vspace{3pt}\\
\!\forall\,(x,p)\!\in\!\overline{\Omega}\!\times\!\R^N,\ g(x,\!\cdot,p)\!>\!0\hbox{ in $\!(0,+\infty)$},\ g(x,\!\cdot,p)\hbox{ is non-decreasing in $\![0,+\infty)$},\vspace{3pt}\\
\!\displaystyle\frac{g(x,s,p)}{s}\mathop{\longrightarrow}_{s\to0^+}+\infty\ \hbox{ and }\ \frac{g(x,s,p)}{s}\mathop{\longrightarrow}_{s\to+\infty}0\ \hbox{ uniformly in $(x,p)\in\overline{\Omega}\times\R^N$}.\eaa\right.
\ee
From the elliptic regularity theory and bootstrap arguments, any solution $u$ of~\eqref{eq}, understood first in the weak $H^1_0(\Omega)$ sense, is actually in $H^2(\Omega)$ and then in $W^{2,q}(\Omega)$ for every $1\le q<\infty$, whence in $C^{1,\beta}(\overline{\Omega})$ for every $\beta\in(0,1)$. It is also in $C^{2,\alpha}_{loc}(\Omega)$, for $\alpha\in(0,1)$ such that $x\mapsto g(x,u(x),\nabla u(x))$ is of class $C^{0,\alpha}(\overline{\Omega})$. In particular, the solutions are then understood as classical solutions throughout the paper. Typical examples of functions $g:\overline{\Omega}\to[0,+\infty)\times\R^N\to\R$ satisfying~\eqref{hypg} are given by
\begin{itemize}
\item $g:(x,s,p)\mapsto g(x,s,p)$ bounded, locally H\"older-continuous in $\overline{\Omega}\times[0,+\infty)\times\R^N$, non-decreasing in $s$, and $\inf_{\overline{\Omega}\times[0,+\infty)\times\R^N}g>0$;
\item $g(x,s,p)=a(x,s,p)s^r+b(x,s,p)$ with $0<r<1$, $a$ and $b$ bounded, locally H\"older-continuous in $\overline{\Omega}\times[0,+\infty)\times\R^N$, non-decreasing in $s$, $b$ non-negative, and $\inf_{\overline{\Omega}\times[0,+\infty)\times\R^N}a>0$;
\item $g(x,s,p)=a(x,s,p)\ln(1+b(x,s,p)s^r)$ with $0<r<1$, $a$ and $b$ bounded, locally H\"older-continuous in $\overline{\Omega}\times[0,+\infty)\times\R^N$, non-decreasing in $s$, and $\inf_{\overline{\Omega}\times[0,+\infty)\times\R^N}a>0$, $\inf_{\overline{\Omega}\times[0,+\infty)\times\R^N}b>0$;
\item $g(x,s,p)=\ln(1+a(x,s,p)s^r)+b(x,s,p)$ with $r>0$, $a$ and $b$ bounded, locally H\"older-continuous in $\overline{\Omega}\times[0,+\infty)\times\R^N$, non-decreasing in $s$, $a$ non-negative, and $\inf_{\overline{\Omega}\times[0,+\infty)\times\R^N}b>0$ (this case actually contains the first one, by choosing $a\equiv0$),
\end{itemize}
or any linear combination with positive coefficients of any number of functions of the above types.

Observe that, if $u$ solves~\eqref{eq}, then
\be\label{intepsdelta}
\epsilon\int_\Omega u(x)\,\varphi_1(x)\,dx=\delta\int_\Omega g(x,u(x),\nabla u(x))\,\varphi_1(x)\,dx.
\ee
Since both integrals in the above formula are positive, one necessarily has
$$\mathrm{sign}(\epsilon)=\mathrm{sign}(\delta),$$
where, for a real number $x$, $\mathrm{sign}(x):=x/|x|$ if $x\neq0$ and $\mathrm{sign}(0):=0$. When $\epsilon=\delta=0$, equation~\eqref{eq} amounts to~\eqref{varphi1}, and the solutions are the positive multiples of any given principal eigenfunction $\varphi_1$ of~\eqref{varphi1}. For~\eqref{varphi1} again, it is well-known since the seminal paper of Brascamp and Lieb~\cite{bl} that, if $\Omega$ is convex, then $\varphi_1$ is log-concave, that is, $\log\varphi_1$ is concave (see further comments on the literature after Theorem~\ref{th1}).

Our main result shows that~\eqref{eq} admits solutions when $\mathrm{sign}(\delta)=\mathrm{sign}(\epsilon)$ and when~$\epsilon$ is either negative, or positive and small (with $\delta$ and $\epsilon$ of the same order in the latter case), that all solutions are close to a principal eigenfunction $\varphi_1$ of~\eqref{varphi1} when $\epsilon$ and $\delta$ are small in absolute values and of the same order, and that these solutions are log-concave when $\Omega$ is smooth and strictly convex (in the sense that all principal curvatures on $\partial\Omega$ are positive).

\begin{theorem}\label{th1}
Let $g$ satisfy~\eqref{hypg}, and let $\varphi_1$ be a given principal eigenfunction of~\eqref{varphi1}. The following properties hold:
\begin{enumerate}
\item [(i)] for every $\epsilon<0$ and $\delta<0$, problem~\eqref{eq} has at least one solution; furthermore, for each $A\ge1$, there is $\epsilon_0>0$ such that, for every $\epsilon\in(0,\epsilon_0)$ and $A^{-1}\epsilon\le\delta\le A\epsilon$, problem~\eqref{eq} has at least one solution;
\item [(ii)] if the function $(x,s,p)\mapsto g(x,s,p)=g(x,s)$ does not depend on $p$ and
\be\label{hypg2}
\forall\,x\in\Omega,\ \ \ s\mapsto\frac{g(x,s)}{s}\ \hbox{ is decreasing in $(0,+\infty)$},
\ee
and if there is $c>0$ such that $\delta\sim c\,\epsilon$ as $\dsp\epsilon{\mathop{\to}^{\neq}}0$, then there is a constant $B=B(\Omega,g,c,\varphi_1)>0$ such that $u_{\epsilon,\delta}\to B\varphi_1$ in $C^1(\overline{\Omega})$ as $\dsp\epsilon\mathop{\to}^{\neq}0$, for all solutions $u_{\epsilon,\delta}$ of~\eqref{eq};
\item [(iii)] if $\partial\Omega$ is of class $C^{2,\alpha}$ for some $\alpha\in(0,1)$ and if $\Omega$ is strictly convex, then, for each $A\ge1$, there is $\epsilon'_0>0$ such that all solutions of~\eqref{eq} are log-concave when $|\epsilon|\le\epsilon'_0$, $A^{-1}|\epsilon|\le|\delta|\le A|\epsilon|$ and $\mathrm{sign}(\delta)=\mathrm{sign}(\epsilon)$.
\end{enumerate}
\end{theorem}

\subsection*{Comments on concavity properties}

Several comments are in order. Let us start with the concavity properties. Throughout these comments, we assume that $\Omega$ is convex. It is natural to wonder whether the solutions~$u$ of~\eqref{eq} inherit a sort of convexity. Remembering that they are positive in $\Omega$ and vanish on $\partial\Omega$, one could wonder whether they would be concave. But it is well-known~\cite{ka2} that, for instance, the solutions $\varphi_1$ of~\eqref{varphi1} are not concave in general. However, for the solutions $\varphi_1$ of~\eqref{varphi1} or for any solution $u$ of~\eqref{eq} in a convex domain $\Omega$, the upper level $\{x\in\Omega:u(x)>0\}=\Omega$ is convex. What about the other upper level sets
$$\{x\in\Omega:u(x)>\lambda\},$$
for $\lambda>0$~? When they are all convex, the function $u$ is said to be quasi-concave. For instance, if $\Omega$ is an Euclidean ball, if $f:[0,+\infty)\to\R$ is locally Lipschitz-continuous and if $u$ is a solution of
\be\label{eqf}\left\{\baa{rcll}
\Delta u+f(u) & \!\!=\!\! & 0 & \!\hbox{in $\Omega$},\vspace{3pt}\\
u & \!\!>\!\! & 0 & \!\hbox{in $\Omega$},\vspace{3pt}\\
u & \!\!=\!\! & 0 & \!\hbox{on $\partial\Omega$},\eaa\right.
\ee
then $u$ is radially decreasing with respect to the center of the ball~\cite{gnn}, and its upper level sets are then balls, therefore $u$ is immediately quasi-concave~! But as such this does not imply much more, that is, it is not clear that $h(u)$ would be concave for some increasing function $h:(0,+\infty)\to\R$. Actually, even in some convex domains having some symmetries, the quasi-concavity of $u$ is not true in general for~\eqref{eqf}~\cite{hns}, or for more general quasilinear equations~\cite{aags}.

Coming back to the eigenvalue problem~\eqref{varphi1} in a convex domain~$\Omega$, the function $\log\varphi_1$ is concave~\cite{bl}, hence $\varphi_1$ is quasi-concave. The proof in~\cite{bl} is based on Pr\'ekopa-Leindler and Brunn-Minkowski inequalities, leading to the preservation of the log-concavity by the heat flow. Another proof of the quasi-concavity of $\varphi_1$ in dimension $N=2$ based on PDE arguments was given by Acker, Payne and Philippin~\cite{app}, while another proof of the log-concavity of $\varphi_1$ was given by Caffarelli and Spruck~\cite{casp} and Korevaar~\cite{ko1}, based on the maximum principle applied to the concavity function
$$(x,y,\mu)\mapsto \mu v(x)+(1-\mu)v(y)-v(\mu x+(1-\mu)y),$$
defined in $\Omega\times\Omega\times[0,1]$, with $v:=\log\varphi_1$. Furthermore, based on the constant-rank method, Caffarelli and Friedman~\cite{cf} showed that $\log\varphi_1$ is strictly concave if $\Omega$ is strictly convex, in dimension $N=2$. The strict log-concavity in the sense of the uniform strict negativity of the Hessian matrices of $\log\varphi_1$ was shown by Lee and V\'azquez~\cite{lv} in any dimension $N\ge2$. The concavity of $\log\varphi_1$ implies immediately the concavity of the function $-\varphi_1^{-q}$ for every $q>0$. On the other hand, one can wonder whether $\varphi_1$ would be stronger than log-concave, that is, whether there would exist some increasing function $h:(0,+\infty)\to\R$ such that $h(\varphi_1)$ would be concave and the concavity of $h(\varphi_1)$ would imply the concavity of $\log\varphi_1$. To our knowledge, the answer is actually not known. For the Cauchy problem of the heat equation $v_t=\Delta v$ with Dirichlet boundary condition on $\partial\Omega$, still assuming that the domain $\Omega$ is convex, the log-concavity is preserved, that is, if the initial condition $v_0:=v(0,\cdot)$ is bounded, positive and log-concave in $\Omega$, then so is $v(t,\cdot)$ for every $t>0$~\cite{bl}. Since there is $c>0$ such that
$$e^{\lambda_1t}v(t,\cdot)\mathop{\longrightarrow}_{t\to+\infty}c\,\varphi_1\ \hbox{ uniformly in $\overline{\Omega}$}$$
and since the log-concavity is preserved by multiplication by positive scalars, the log-concavity of $\varphi_1$ follows. It turns out that, assuming without loss of generality that $0<v_0\le1$ in $\Omega$, if $-(-\log v_0)^q$ is concave for some $q\in[1/2,1]$, then $-(-\log v(t,\cdot))^q$ remains concave for every $t>0$, and $q=1/2$ is optimal, as shown by Ishige, Salani and Takatsu~\cite{ist} (and the log-concavity, i.e. with $q=1$, is in some sense the weakest concavity preserved by the heat flow in any dimension~\cite{ist2}, while the quasi-concavity is preserved by the heat flow only in dimension $N=1$~\cite{is,ist2}). Since the concavity of $-(-\log u)^q$ is not preserved by multiplication of $u$ by positive constants when $q\neq1$, the above result does not imply at once that $-(-\log(\varphi_1/\|\varphi_1\|_\infty))^q$ would be concave for $q<1$, though this property is conjectured for $q\in(1/2,1)$ , and immediately true for $q>1$.

For the semilinear equation~\eqref{eqf}, the solutions $u$ are log-concave provided that $f$ is~of class $C^1([0,+\infty))$ and the functions
$$s\mapsto\frac{f(s)}{s}\ \hbox{ and }\ s\mapsto f'(s)-\frac{f(s)}{s}$$
are non-increasing in $(0,+\infty)$~\cite{ko1,li}. Notice that, for~\eqref{eq} with $\epsilon$ small but not signed and $g(x,s,p)=g(s)$, the above conditions would imply that $g(s)=as$ for some $a\in\R$, that is, problem~\eqref{eq} would be linear and the positive solutions of~\eqref{eq} would necessarily be multiples of $\varphi_1$. Lions' method~\cite{li} is based on the preservation of the log-concavity by the heat flow~\cite{bl},\footnote{The log-concavity is preserved for more general parabolic equations $v_t=\Delta v+f(v)$, based on maximum principles for the concavity function of the logarithm~\cite{gk,ko1}, but it is not for linear parabolic equations with non-constant second-order coefficients~\cite{kol}.} on the Trotter formula, and on the convergence of the positive solutions of the corresponding Cauchy problems to the solution of~\eqref{eqf}, which is unique and exists under the additional assumptions $\lim_{s\to0^+}f(s)/s>\lambda_1$ and $f(s)\le0$ for large $s$, by~\cite{b}. The log-concavity is also known for~\eqref{eqf} with $f(s)=s\log s^2$~\cite{gms}, for the principal eigenfunctions of elliptic operators with drift terms~\cite{c,cfls}, for $p$-Laplace equations~\cite{bms,cqs,sak}, for equations with $p$-normalized Laplace operators~\cite{ku}, or for the Laplace operator with Robin boundary conditions and large Robin parameter~\cite{crfr}. For the positive solutions~$u$ of equations of the type
$$\mathrm{div}(A(x)\nabla u)+a(x)u-\epsilon h(u)=0$$
with an non-decreasing function $h\in C^1((0,+\infty))$, Almousa, Bucur, Cornale and Squassina~\cite{abcs} showed that $\log u$ is $O(\epsilon)$-close in $L^\infty(\Omega)$ to a concave function as $\epsilon\to0^+$, when $A$ and $a$ have gradients of order $O(\epsilon)$. The method relies on estimates of the concavity functions associated to $\log u$, for possible positive solutions $u$. Under different assumptions on the coefficients of the equation, Theorem~\ref{th1} shows the existence of positive solutions for~\eqref{eq}, and exact log-concavity properties when $|\epsilon|$ and $|\delta|$ are small. The method used here is also different from~\cite{abcs} since, in order to show the log-concavity of the solutions in part~$(iii)$ of Theorem~\ref{th1}, we show, through a priori estimates and uniform lower and upper bounds in the limit $(\epsilon,\delta)\to(0,0)$, that the solutions actually converge in $C^2(\overline{\Omega})$ to positive multiples of $\varphi_1$, and we invoke the strict log-concavity of $\varphi_1$~\cite{lv} to conclude. We add that the strict log-concavity implies in particular that the upper level sets of the solutions are strictly convex and that the solutions have exactly one critical point, their maximum (we refer to~\cite{gro} for a survey on the number of critical points of semilinear elliptic equations).

Let us finally mention other results on the quasi-concavity of solutions $u$ of equations of the type~\eqref{eqf} in bounded convex domains $\Omega$. When $f$ is a positive constant, the celebrated result of Makar-Limanov~\cite{ml} states that, in dimension $N=2$, $\sqrt{u}$ is concave, hence $u$ is log-concave and quasi-concave. More generally speaking, in any dimension $N\ge2$, when $f(s)=\mu\,s^p$ with $p\in[0,1)$ and $\mu>0$, then any solution $u$ of~\eqref{eqf} is such that $u^{(1-p)/2}$ is concave~\cite{k,ke1}, and even strictly concave if $\Omega$ is strictly convex~\cite{lv}. In addition to the previous references, further results on the concavity of $h(u)$, for some function $h:(0,+\infty)\to\R$, have been obtained under various assumptions on $f$, see e.g.~\cite{abcs,casp,gp,ka3,ka4,ko1} with arguments based on the maximum principle for a certain concavity function (in addition to~\cite{bms,sak} for $p$-Laplace equations), or~\cite{cf,gms,kl,l,st,x} with arguments based on properties of the rank of the Hessian matrices of the solutions or of functions of them. Other quasi-concavity results for the solutions of some equations~\eqref{eqf} have been obtained in~\cite{f,gg,ka2}.
 
\subsection*{Comments of the existence results, and the link with maximum and anti-maximum principles}

In these comments, $\Omega$ is only assumed to be a bounded domain with $C^{1,1}$ boundary, but it is not assumed to be convex. First of all, we refer to the celebrated paper of Ambrosetti, Brezis and Cerami~\cite{abc} for some existence results of solutions of equations of the type~\eqref{eqf}. For equation~\eqref{eq}, the existence and convergence properties strongly depend on the signs of the parameters $\epsilon$ and $\delta$ (we recall that they are necessarily of the same sign). Consider first the easier case when $\epsilon$ and $\delta$ are {\it negative}. In that case, on the one hand, from~\eqref{hypg}, the function~$\sigma\varphi_1$ is a sub-solution of~\eqref{eq} for all $\sigma>0$ small enough. On the other hand, for a certain well-chosen $a>0$, the solution $\psi$ of $\Delta\psi+(\lambda_1+\epsilon/2)\psi=\delta a$ in $\Omega$ with Dirichlet boundary conditions, will be a super-solution of~\eqref{eq} such that $\psi\ge\sigma\varphi_1$ for all $\sigma>0$ small enough (see Section~\ref{sec3} for more details). Now, if the function $g(x,s,p)$ in~\eqref{hypg} were independent of $p$, then these sub- and super-solutions would immediately lead to the existence of a solution $u$ of~\eqref{eq} between them, from a standard iteration process based on the maximum principle~\cite{a,d,pw,sa}. In the general case of equation~\eqref{eq} with an $(x,s,p)$-dependent function $g(x,s,p)$ satisfying~\eqref{hypg}, the above-described strategy for the existence of a solution will be adapted, especially with the help of the Schauder fixed point theorem (the proof presented here provides an alternate proof for~\eqref{eq} of a more general result of De Coster and Henrard~\cite[Theorem~3.1]{dch} based on the topological degree). In order to get part~$(ii)$ of Theorem~\ref{th1} and the convergence of the solutions to positive multiples of $\varphi_1$ as $(\epsilon,\delta)\to(0^-,0^-)$ with~$\epsilon$ and~$\delta$ of the same order, some a priori integral and pointwise bounds are derived. 

The situation is radically different when $\epsilon$ and $\delta$ are {\it positive}. Now, due to~\eqref{hypg}, $\sigma\varphi_1$ is a super-solution of~\eqref{eq} for every $\sigma>0$ small enough, while the constant function~$M$ is a sub-solution for every $M>0$ large enough. So the standard method of sub- and super-solution is not applicable at once. We use the Schauder fixed point argument (actually as in the case where $\epsilon$ and $\delta$ are negative), but instead of making use of the maximum principle when $\epsilon<0$, the construction of a convex set that is stable by a compact map now relies on a quantified version of the anti-maximum principle, when $\epsilon>0$ is small enough. The standard anti-maximum principle of Cl\'ement and Peletier~\cite{cp} (see also~\cite{sh}) asserts that, if $w\in L^p(\Omega)$ with $p\in(N,+\infty)$ and $w>0$ in~$\Omega$, then there is $\epsilon_w>0$ such that, for any $\epsilon\in(0,\epsilon_w)$, the solution $\mathcal{T}_\epsilon(w)$ of
\be\label{Teps}\left\{\baa{rcll}
\Delta\mathcal{T}_\epsilon(w)+(\lambda_1+\epsilon)\mathcal{T}_\epsilon(w) & = & w & \hbox{in }\Omega,\vspace{3pt}\\
\mathcal{T}_\epsilon(w) & = & 0 & \hbox{on }\partial\Omega,\eaa\right.
\ee
satisfies $\mathcal{T}_\epsilon(w)>0$ in $\Omega$. The function $\mathcal{T}_\epsilon(w)$ is the $W^{2,p}(\Omega)\cap W^{1,p}_0(\Omega)$ solution of this problem, it exists and is unique if $0<\epsilon<\lambda_2-\lambda_1$ (and also if $\epsilon<0$ or more generally if $\lambda_1+\epsilon$ is different from the eigenvalues $(\lambda_n)_{n\ge1}$~\cite{HB}). The anti-maximum principle also holds for the Laplacian with an indefinite weight having non-trivial positive part in factor of~$\lambda_1$ and with $w$ such that $\int_\Omega w\varphi_1>0$ (in particular,~$w$ may not be positive everywhere in~$\Omega$) as shown by Arcoya and G\'amez~\cite{ag} (see also Fleckinger, Hern\'andez and De Th\'elin~\cite{fht}). The anti-maximum principle is still valid for more general operators~\cite{cp,p,t}, including $p$-Laplacian operators~\cite{bt,fgtt}, for Steklov problems~\cite{apr,apr2}, and, as shown by Birindelli~\cite{bi}, it also holds in less regular domains $\Omega$ with $\mathcal{T}_\epsilon(w)$ and $\varphi_1$ satisfying the boundary condition in the weaker sense of Berestycki, Nirenberg and Varadhan~\cite{bnv}. Furthermore, the inequality $p>N$ is actually sharp~\cite{sw}, and the size $\epsilon_w>0$ of the interval of validity of the anti-maximum principle is not uniform with respect to $w$~\cite{bt,cp} (it is nevertheless uniform in dimension $N=1$ under Neumann boundary conditions~\cite{cp}, or under a more general assumption for higher-order operators~\cite{cs}, periodic operators~\cite{cmo}, or $p$-Laplacian operators in dimension~$N=1$ with Neumann boundary conditions~\cite{r}).

To make our argument work in the limit as $\epsilon\to0^+$, we prove a quantified version of the anti-maximum principle when the right-hand sides of~\eqref{Teps} have integrals against $\varphi_1$ of order $\epsilon$. To do so, for a given $\varphi_1$ solving~\eqref{varphi1}, for any two real numbers $\theta\le\tau$, we define the set
$$E_{\theta,\tau}:=\Big\{w\in L^1(\Omega):\theta\le\int_\Omega w\varphi_1\le\tau\Big\}.$$

\begin{pro}\label{pro2}
Let $p\in(N,+\infty)$, $M\in(0,+\infty)$, and let $\varphi_1$ be the solution of~\eqref{varphi1} such that $\|\varphi_1\|_{L^2(\Omega)}=1$. Then, for every $0<a\le b<+\infty$,
$$a\le\liminf_{(\epsilon,\eta)\to(0^+,0^+)}\Big[\inf_{w\in E_{\epsilon a,\epsilon b},\,\|w\|_{L^1(\Omega)}\le\eta,\,\|w\|_{L^p(\Omega)}\le M}\Big(\inf_{\Omega}\frac{\mathcal{T}_\epsilon(w)}{\varphi_1}\Big)\Big]$$
and
$$\limsup_{(\epsilon,\eta)\to(0^+,0^+)}\Big[\sup_{w\in E_{\epsilon a,\epsilon b},\,\|w\|_{L^1(\Omega)}\le\eta,\,\|w\|_{L^p(\Omega)}\le M}\Big(\sup_{\Omega}\frac{\mathcal{T}_\epsilon(w)}{\varphi_1}\Big)\Big]\le b.$$
\end{pro}

In particular, for a single fixed function $f\in L^p(\Omega)$ with $p>N$ and $a:=\int_\Omega f\varphi_1>0$, then $\epsilon f\in E_{\epsilon a,\epsilon a}$ and, for every $\epsilon>0$ such that $\epsilon\|f\|_{L^p(\Omega)}\le1$,
\be\label{ineq1}
\epsilon\,\inf_\Omega\frac{\mathcal{T}_\epsilon(f)}{\varphi_1}=\inf_\Omega\frac{\mathcal{T}_\epsilon(\epsilon f)}{\varphi_1}\ge\inf_{w\in E_{\epsilon a,\epsilon a},\,\|w\|_{L^1(\Omega)}\le\epsilon\|f\|_{L^1(\Omega)},\,\|w\|_{L^p(\Omega)}\le 1}\Big(\inf_{\Omega}\frac{\mathcal{T}_\epsilon(w)}{\varphi_1}\Big),
\ee
whence Proposition~\ref{pro2} implies that
$$\inf_\Omega\frac{\mathcal{T}_\epsilon(f)}{\varphi_1}\to+\infty\ \hbox{ as }\epsilon\to0^+.$$
This last property could also be viewed as consequences of some results of~\cite{cp}, or of~\cite{kor} when $f>0$ with a method based on the implicit function theorem, and or of a result of Hess~\cite{h}, which holds for more general operators with indefinite weight and functions $f$ positive and continuous in $\overline{\Omega}$. As a matter of fact, the right-hand side of~\eqref{ineq1} converges to $a$ as $\epsilon\to0^+$ from Proposition~\ref{pro2}, and a similar upper bound finally entails that
$$\epsilon\,\frac{\mathcal{T}_\epsilon(f)}{\varphi_1}\to\int_\Omega f\,\varphi_1\ \hbox{ uniformly in $\Omega$ as $\epsilon\to0^+$}.$$

Proposition~\ref{pro2} is shown through a spectral decomposition and some a priori estimates. For the proof of part~$(i)$ of Theorem~\ref{th1} with $\epsilon>0$ and $\delta>0$ small enough and such that $A^{-1}\epsilon\le\delta\le A\epsilon$, the Schauder fixed point theorem can be applied in the set~$\mathcal{C}$ of $C^1(\overline{\Omega})$ functions $v$ such that $\sigma\varphi_1\le v\le\beta$, for some well-chosen $\sigma>0$ and $\beta>0$ depending on $A$. To make sure that $v\mapsto\mathcal{T}_\epsilon(\delta\,g(\cdot,v,\nabla v))$ maps $\mathcal{C}$ compactly into itself, one uses the assumptions on $g$ and Proposition~\ref{pro2}. Thus, the existence of solutions of~\eqref{eq} follows. Then, elliptic estimates and integral bounds lead to the convergence of the solutions $u_{\epsilon,\delta}$ to positive multiples of $\varphi_1$ as $(\epsilon,\delta)\to(0^+,0^+)$ when~$\epsilon$ and~$\delta$ are of the same order.

For~\eqref{eq} with $\delta g(x,u)$ instead of $\delta g(x,u,\nabla u)$, the existence of solutions~$u$ such that $\int_\Omega u\varphi_1\in(a,b)$ with $a<b\in\R$ was shown by Rezende, S\'anchez-Aguilar and Silva~\cite[Theo\-rems~1.1-1.2]{rss}, provided that the integrals $\int_\Omega g(\cdot,a\varphi_1)\varphi_1$ and $\int_\Omega g(\cdot,b\varphi_1)\varphi_1$ have opposite strict sign and that $0<|\epsilon|<\!\!<\delta$. These conditions, which genera\-lize the Landesman-Lazer condition for the solvability of equations $\Delta u+\lambda_1u=h(x,u)$ with Dirichlet boundary condition~\cite{ll}, are actually not satisfied here since the sought solutions are positive and $g>0$ in~$\overline{\Omega}\times(0,+\infty)\times\R^N$ ($g$ could be continuously extended by $g(x,s,p)=g(x,0,p)\ge0$ for all $(x,s,p)\in\overline{\Omega}\times(-\infty,0)\times\R^N$). Nevertheless,~\cite[Theorem~1.5]{rss} (with $\delta g(x,u)$ instead of $\delta g(x,u,\nabla u)$) shows that the conclusions $(i)$-$(ii)$ of Theorem~\ref{th1} do not hold without the condition $A^{-1}|\epsilon|\le|\delta|\le A|\epsilon|$, namely they do not hold when $0<|\epsilon|<\!\!<|\delta|<\!\!<1$. As a matter of fact, independently of~\cite[Theorem~1.5]{rss}, since $g$ is here continuous in~$\overline{\Omega}\times[0,+\infty)\times\R^N$ and positive in $\overline{\Omega}\times(0,+\infty)\times\R^N$, and since formula~\eqref{intepsdelta} holds, the convergence of solutions $u$ of~\eqref{eq} to $B\varphi_1$ in $C^1(\overline{\Omega})$ as $(|\epsilon|,|\delta|)\to(0^+,0^+)$ with $B>0$ necessarily entails that both ratios $|\epsilon|/|\delta|$ and $|\delta|/|\epsilon|$ must be bounded.

\subsection*{Some open problems}

The proof of Theorem~\ref{th1} relies on Proposition~\ref{pro2} and on some compactness arguments, among other things. These arguments provide the existence of positive real numbers~$\epsilon_0$ and~$\epsilon'_0$ for the existence in part~$(i)$ and for the log-concavity in part~$(iii)$. A general question then arises:

\begin{open}
Can one derive some lower bounds on $\epsilon_0$ in part~$(i)$ of Theorem~$\ref{th1}$ and on $\epsilon'_0$ in part~$(iii)$, in terms of $g$, $\Omega$, and $A$~?
\end{open}

Actually, some lower bounds of the interval of validity of the anti-maximum principle are known for other boundary conditions, such as Neumann~\cite{cp,cs} or periodic boundary conditions~\cite{cmo}, while some more explicit $w$-dependent lower estimates on the size of this interval for~\eqref{Teps} are given in~\cite{fht}, and some upper estimates are provided in~\cite{bdi}, even for $p$-Laplacian operators. 

In the case $\delta<0$, the assumption~\eqref{hypg2} guarantees the uniqueness of the solutions~$u$ to~\eqref{eq}, from the papers of Berestycki~\cite{b} and Brezis and Oswald~\cite{bo}. The proof strongly uses the fact that the function $s\mapsto-\delta g(x,s)/s$ is decreasing in $(0,+\infty)$ for each $x\in\Omega$. When $\delta>0$, the same condition cannot hold, since it is incompatible with~\eqref{hypg}, which says that $g(x,s)/s\to+\infty$ as $s\to0^+$ uniformly in $x\in\Omega$. However, part~$(ii)$ of Theorem~\ref{th1} states that, under condition~\eqref{hypg2}, the solutions of~\eqref{eq}, even if they were not unique when~$\epsilon$ and~$\delta$ are positive, converge to a {\it unique} multiple $B$ of a given eigenfunction~$\varphi_1$ as~$\epsilon\to0^{\pm}$ and $\delta\sim c\epsilon$ with $c>0$. A natural question then arises.

\begin{open}
Under the conditions~\eqref{hypg} and~\eqref{hypg2}, is it true that the solutions of~\eqref{eq} are unique when $\epsilon$ and $\delta$ are positive~?
\end{open}

In the proof of part~$(i)$ of Theorem~\ref{th1} in the case $\mathrm{sign}(\epsilon)=\mathrm{sign}(\delta)=1$, thanks to~\eqref{hypg}, some positive super-solutions $\overline{u}$ of~\eqref{eq} and positive sub-solutions $\underline{u}$ of~\eqref{eq} are constructed with the unusual inequality $\overline{u}\le\underline{u}$ in $\overline{\Omega}$. The existence of a solution $u$ of~\eqref{eq} between $\overline{u}$ and $\underline{u}$ is shown by using a priori estimates, Proposition~\ref{pro2}, and the Schauder fixed point theorem. Whereas, for the case of an equation~$\Delta u+f(x,u)=0$ in $\Omega$, the existence of a solution $u$ between $\underline{u}$ and $\overline{u}$ in the case $\underline{u}\le\overline{u}$ is standard~\cite{a}, the case $\overline{u}\le\underline{u}$ does not work as such. Actually, the existence of $u$ such that $\overline{u}\le u\le\underline{u}$ is not true in general~\cite[p. 653]{a}. However, for the equation $u''+f(x,u)=0$ in a bounded interval with Dirichlet boundary conditions, some sufficient conditions on $f(x,s)$, with $f$ concave in $s$ and truly heterogeneous in $x$, lead to the existence of a solution $u$ such that $\overline{u}\le u\le\underline{u}$ in a bounded interval~\cite{kor1}. For the equation~\eqref{eqf} without $x$ dependence, we mention the following open question (this problem, originating from~\cite{sa}, is still open to our knowledge, despite on some results assuming partial or complete unusual ordering between the sub- and super-solutions and leading to the existence of a solution satisfying partial ordering between the sub- and super-solutions~\cite{dchb,dch,go,ho}):

\begin{open}
Given a Lipschitz-continuous function $f:[0,+\infty)\to\R$, and some $C^2(\Omega)\cap C(\overline{\Omega})$ functions $\underline{u}$ and $\overline{u}$ such that $\Delta\overline{u}+f(\overline{u})\le0\le\Delta\underline{u}+f(\underline{u})$ in $\Omega$, $\underline{u}=\overline{u}=0$ on $\partial\Omega$, and $\overline{u}\le\underline{u}$ in $\overline{\Omega}$, is there a solution $u$ to~\eqref{eqf} such that $\overline{u}\le u\le\underline{u}$ in $\overline{\Omega}$, under some general conditions on $f$~?
\end{open}

As in the previous paragraph, in part~$(i)$ of Theorem~\ref{th1} in the case when $\epsilon$ and $\delta$ are negative, without dependence on $\nabla u$ in $g$, there are positive sub- and super-solutions $\underline{u}\le\overline{u}$, whence there exists a solution $u$ in between. The existence of a solution between the sub- and super-solutions could also be derived from the maximum principle applied to the associated parabolic equation starting with a sub- or super-solution as initial condition, and the so-constructed solution to the parabolic equation converges monotonically as $t\to+\infty$ to a semi-stable solution $u$ of~\eqref{eq}. This semi-stable solution is then the minimal or maximal solution between $\underline{u}$ and $\overline{u}$~\cite{abc,sa}. For an equation of the type~\eqref{eqf}, and assuming that~$f$ is of class~$C^1$, the semi-stability of $u$ means that
$$\int_{\Omega}|\nabla\phi|^2-\int_{\Omega}f'(u)\phi^2\ge0$$
for every~$C^{\infty}(\Omega)$ function $\phi$ with compact support included in $\Omega$. Cabr\'e and Chanillo~\cite{cc} showed that, if $\Omega$ is $C^\infty$ and strictly convex in dimension $N=2$, and if $f$ is $C^\infty$ and non-negative, then a semi-stable solution $u$ of~\eqref{eqf} has a unique critical point (its maximum point), which is non-degenerate. Therefore, $u$ is concave in the neighborhood of its maximum point. On the other hand, for some $f$ and some specific stadium-like domains in dimension $N=2$, there are some solutions $u$ of~\eqref{eqf} which are not even quasi-concave~\cite{hns}. It is expected that these non-quasi-concave solutions are not semi-stable. But this is still an open question, which can be formulated as follows.

\begin{open}
Assume that $\Omega$ is convex, in any dimension $N\ge2$, and $f:[0,+\infty)\to\R$ is Lipschitz-continuous. Is it true that, if $u$ is a semi-stable solution of~\eqref{eqf}, then $u$ is quasi-concave~?
\end{open}

Notice that the result is immediately true in dimension $N=1$, or more generally if~$\Omega$ is an Euclidean ball in any dimension $N\ge1$, even without the semi-stability condition, since $u$ is necessarily radially symmetric and decreasing with respect to the center of the ball.


\section{Quantified anti-maximum principle: proof of Proposition~\ref{pro2}}\label{sec2}

Let $p\in(N,+\infty)$, $M\in(0,+\infty)$, $0<a\le b<+\infty$, and $\varphi_1$ solve~\eqref{varphi1} with $\|\varphi_1\|_{L^2(\Omega)}=1$. Let $(\epsilon_n,\eta_n)_{n\in\N}$ be any sequence converging to $(0^+,0^+)$. Without loss of generality, one can assume that
\be\label{hypepsn}
0<\epsilon_n<\frac{\lambda_2-\lambda_1}{2}
\ee
for all $n\in\N$. Let $(w_n)_{n\in\N}$ be a sequence in $L^p(\Omega)$ such that
\be\label{hypwn}
\|w_n\|_{L^p(\Omega)}\le M,\ \ \|w_n\|_{L^1(\Omega)}\le\eta_n\ \hbox{ and }\ \epsilon_na\le\int_\Omega w_n\varphi_1\le\epsilon_nb.
\ee
For each $n\in\N$, let $v_n:=\mathcal{T}_{\epsilon_n}(w_n)$, that is,
\be\label{eqvn}\left\{\baa{rcll}
\Delta v_n+(\lambda_1+\epsilon_n)v_n & = & w_n & \hbox{in }\Omega,\vspace{3pt}\\
v_n & = & 0 & \hbox{on }\partial\Omega.\eaa\right.
\ee
We recall that $v_n\in W^{2,p}(\Omega)\cap W^{1,p}_0(\Omega)$, whence $v_n\in C^1(\overline{\Omega})$. The goal is to compare $v_n$ with multiples of $\varphi_1$ for all $n$ large enough.

First of all, by multiplying~\eqref{eqvn} by $\varphi_1$ and integrating by parts together with~\eqref{varphi1}, one gets that
$$\epsilon_n\int_\Omega v_n\varphi_1=\int_\Omega w_n\varphi_1=:\tau_n,$$
whence
\be\label{intvn}
a\le\theta_n:=\int_\Omega v_n\varphi_1\le b
\ee
for all $n\in\N$. Secondly, thanks to~\eqref{varphi1} again, problem~\eqref{eqvn} can be rewritten as
$$\left\{\baa{rcll}
\Delta(v_n-\theta_n\varphi_1)+(\lambda_1+\epsilon_n)(v_n-\theta_n\varphi_1) & = & w_n-\epsilon_n\theta_n\varphi_1 & \hbox{in }\Omega,\vspace{3pt}\\
v_n-\theta_n\varphi_1 & = & 0 & \hbox{on }\partial\Omega.\eaa\right.$$
By multiplying the above equation by $v_n-\theta_n\varphi_1\in W^{2,p}(\Omega)\cap W^{1,p}_0(\Omega)\cap C^1(\overline{\Omega})$ and integrating by parts, one gets that
\be\label{vn0}
\int_\Omega|\nabla(v_n-\theta_n\varphi_1)|^2-(\lambda_1+\epsilon_n)\int_\Omega(v_n-\theta_n\varphi_1)^2=-\int_\Omega(w_n-\epsilon_n\theta_n\varphi_1)\,(v_n-\theta_n\varphi_1).
\ee
But each function $v_n-\theta_n\varphi_1\in H^1_0(\Omega)$ is orthogonal to $\varphi_1$ for the scalar product in $L^2(\Omega)$, since $\|\varphi_1\|_{L^2(\Omega)}=1$. From the decomposition of $L^2(\Omega)$ with respect to the basis of eigenfunctions of the Laplace operator with Dirichlet boundary conditions~\cite{HB}, it follows that
\be\label{vn1}
\int_\Omega|\nabla(v_n-\theta_n\varphi_1)|^2\ge\lambda_2\int_\Omega(v_n-\theta_n\varphi_1)^2
\ee
for all $n\in\N$, whence
\be\label{vn2}\baa{rcl}
\displaystyle\int_\Omega|\nabla(v_n\!\!-\theta_n\varphi_1)|^2-(\lambda_1\!+\!\epsilon_n)\int_\Omega(v_n\!-\!\theta_n\varphi_1)^2 & \ge & \displaystyle\Big(1\!-\!\frac{\lambda_1\!+\!\epsilon_n}{\lambda_2}\Big)\int_\Omega|\nabla(v_n\!-\!\theta_n\varphi_1)|^2\vspace{3pt}\\
& \ge & \displaystyle\frac{\lambda_2-\lambda_1}{2\lambda_2}\int_\Omega|\nabla(v_n-\theta_n\varphi_1)|^2\eaa
\ee
since $\epsilon_n<(\lambda_2-\lambda_1)/2$. On the other hand, since $p>N$, there holds
$$1<p':=\frac{p}{p-1}<\frac{2N}{N-2}\ \hbox{(if $N\ge3$)},\ \hbox{ and }\ p'<+\infty\ \hbox{(if $N\le2$)}.$$
Therefore, from the Sobolev embeddings and Poincar\'e inequality, there is a constant $C>0$ such that $C\|v\|_{L^{p'}(\Omega)}^2\le\|\nabla v\|_{L^2(\Omega)}^2$ for all $v\in H^1_0(\Omega)$. Together with~\eqref{vn0} and~\eqref{vn2}, one gets that
$$C\frac{\lambda_2-\lambda_1}{2\lambda_2}\|v_n-\theta_n\varphi_1\|_{L^{p'}(\Omega)}^2\le\|w_n-\epsilon_n\theta_n\varphi_1\|_{L^p(\Omega)}\|v_n-\theta_n\varphi_1\|_{L^{p'}(\Omega)},$$
that is,
$$C\frac{\lambda_2-\lambda_1}{2\lambda_2}\|v_n-\theta_n\varphi_1\|_{L^{p'}(\Omega)}\le\|w_n-\epsilon_n\theta_n\varphi_1\|_{L^p(\Omega)},$$
for all $n\in\N$. But the sequence $(w_n)_{n\in\N}$ is bounded in $L^p(\Omega)$, and the sequences $(\epsilon_n)_{n\in\N}$ and~$(\theta_n)_{n\in\N}$ are bounded in $\R$, whence the sequence $(w_n-\epsilon_n\theta_n\varphi_1)_{n\in\N}$ is bounded in~$L^p(\Omega)$. As a consequence, the sequence $(v_n-\theta_n\varphi_1)_{n\in\N}$ is bounded in $L^{p'}(\Omega)$, and so is the sequence~$(v_n)_{n\in\N}$. By rewriting~\eqref{eqvn} as
$$\left\{\baa{rcll}
\Delta v_n & = & w_n-(\lambda_1+\epsilon_n)v_n & \hbox{in }\Omega,\vspace{3pt}\\
v_n & = & 0 & \hbox{on }\partial\Omega,\eaa\right.$$
the sequence of right-hand sides $(w_n-(\lambda_1+\epsilon_n)v_n)_{n\in\N}$ is bounded in $L^{\min(p,p')}(\Omega)$. By a bootstrap procedure, it follows finally that the sequence $(v_n)_{n\in\N}$ is bounded in $W^{2,p}(\Omega)$.

Therefore, remembering that $p\in(N,+\infty)$, there exists a function $v$ belonging to~$W^{2,p}(\Omega)\cap W^{1,p}_0(\Omega)\,(\subset C^1(\overline{\Omega}))$ such that, up to a subsequence,
$$v_n\to v\ \hbox{ in $C^1(\overline{\Omega})$ as~$n\to+\infty$}.$$
Remember that $\|w_n\|_{L^1(\Omega)}\le\eta_n\to0$ as $n\to+\infty$. By multiplying~\eqref{eqvn} by any test function~$\varphi$ in $C^1(\Omega)$ with compact support included in $\Omega$, and passing to the limit as~$n\to+\infty$, it follows that $v$ is an $H^1_0(\Omega)$ weak solution of~$\Delta v+\lambda_1v=0$ in $\Omega$, and then a classical solution, with Dirichlet boundary conditions. Therefore, from the uniqueness of the eigenfunctions of~\eqref{varphi1} up to multiplicative constants, there is $c\in\R$ such that
$$v=c\varphi_1.$$
Together with~\eqref{intvn} and $\|\varphi_1\|_{L^2(\Omega)}=1$, this entails that $0<a\le c\le b$. Lastly, since $v_n=\varphi_1=0$ and $\nu\cdot\nabla\varphi_1<0$ on $\partial\Omega$, together with $\varphi_1>0$ in $\Omega$ and $v_n\to c\varphi_1$ in~$C^1(\overline{\Omega})$ as $n\to+\infty$, one finally concludes that
$$\frac{v_n}{\varphi_1}\to c\ \hbox{ uniformly in $\Omega$}\ \hbox{ as $n\to+\infty$}.$$
Since $c\in[a,b]$ and the sequences $(\epsilon_n,\eta_n)_{n\in\N}$ converging to $(0^+,0^+)$ and $(w_n)_{n\in\N}$ satis\-fying~\eqref{hypwn} were arbitrary, the proof of Proposition~\ref{pro2} is thereby complete.\hfill$\Box$\break

With similar arguments, the following quantified version of the maximum principle for~\eqref{Teps} when $\epsilon\to0^-$ can be shown. 

\begin{pro}\label{pro3}
Let $p\in(N,+\infty)$, $M\in(0,+\infty)$, and let $\varphi_1$ be the solution of~\eqref{varphi1} such that $\|\varphi_1\|_{L^2(\Omega)}=1$. Then, for every $0<a\le b<+\infty$,
$$a\le\liminf_{(\epsilon,\eta)\to(0^-,0^+)}\Big[\inf_{w\in E_{\epsilon b,\epsilon a},\,\|w\|_{L^1(\Omega)}\le\eta,\,\|w\|_{L^p(\Omega)}\le M}\Big(\inf_{\Omega}\frac{\mathcal{T}_\epsilon(w)}{\varphi_1}\Big)\Big]$$
and
$$\limsup_{(\epsilon,\eta)\to(0^-,0^+)}\Big[\sup_{w\in E_{\epsilon b,\epsilon a},\,\|w\|_{L^1(\Omega)}\le\eta,\,\|w\|_{L^p(\Omega)}\le M}\Big(\sup_{\Omega}\frac{\mathcal{T}_\epsilon(w)}{\varphi_1}\Big)\Big]\le b.$$
\end{pro}

\begin{proof}
We point out that the condition $w\in E_{\epsilon b,\epsilon a}$ means that
$$\epsilon b\le\int_\Omega w\varphi_1\le\epsilon a<0$$
when $\epsilon<0$. The condition~\eqref{hypepsn} used in the proof of Proposition~\ref{pro2} can now be replaced by $-\lambda_1<\epsilon_n<0$. With the same other notations as in the proof of Proposition~\ref{pro2}, the condition~\eqref{intvn} still holds, while the number $2$ in the denominator of the last fraction in~\eqref{vn2} is not needed anymore. The rest of the proof is identical to that of Proposition~\ref{pro2}.
\end{proof}

Proposition~\ref{pro3} will actually not be used in the proof of Theorem~\ref{th1}, unlike Proposition~\ref{pro2}. But, because of its similarity to Proposition~\ref{pro2}, we stated Proposition~\ref{pro3} as a result of independent interest, on a quantified maximum principle in the vicinity of the principal eigenvalue problem~\eqref{varphi1}, for right-hand sides which satisfy some integral estimates but may not be positive.


\section{Proof of Theorem~\ref{th1}, part~$(i)$: existence of solutions}\label{sec3}

As announced in Section~\ref{intro}, the proofs of the existence of solutions to~\eqref{eq} are radically different according to the sign of $\epsilon$ and $\delta$. We also point out that, when $\epsilon=\delta=0$, then problem~\eqref{eq} reduces to~\eqref{varphi1}, for which the existence of positive solutions is known (the principal eigenfunctions). 

\vskip 0.2cm
\noindent{\it First case:}
$$\epsilon<0\ \hbox{ and }\ \delta<0.$$
From~\eqref{hypg}, there is a positive constant $C>0$ such that
\be\label{defC}
\forall\,(x,s,p)\in\overline{\Omega}\times[0,+\infty)\times\R^N,\ \ 0\le g(x,s,p)\le C+\frac{|\epsilon|s}{2|\delta|}.
\ee
Let $\psi$ be the solution of
\be\label{defpsi}\left\{\baa{rcll}
\displaystyle\Delta\psi+\Big(\lambda_1+\frac{\epsilon}{2}\Big)\psi & \!\!=\!\! & \delta C<0 & \!\hbox{in $\Omega$},\vspace{3pt}\\
\psi & \!\!=\!\! & 0 & \!\hbox{on $\partial\Omega$}.\eaa\right.
\ee
Since $\epsilon<0$, the function $\psi$ exists and belongs to all $W^{2,p}(\Omega)\cap W^{1,p}_0(\Omega)$ for $1\le p<+\infty$, and it is unique and positive in $\Omega$, from the maximum principle~\cite{bnv}. It also follows from Hopf's lemma that $\nu\cdot\nabla\psi<0$ on $\partial\Omega$.

From~\eqref{hypg} again, there is $s_0>0$ such that
\be\label{defs0}
\forall\,(x,s,p)\in\overline{\Omega}\times[0,s_0]\times\R^N,\ \ g(x,s,p)\ge\frac{|\epsilon|s}{|\delta|}.
\ee
One can then fix $\sigma>0$ small enough such that
\be\label{defsigma}
\sigma\|\varphi_1\|_{L^\infty(\Omega)}\le s_0\ \hbox{ and }\ \sigma\varphi_1\le\psi\hbox{ in~$\overline{\Omega}$}.
\ee

Consider now the non-empty convex set 
$$E:=\big\{v\in C^1(\overline{\Omega}): \sigma\varphi_1\le v\le\psi\hbox{ in }\overline{\Omega}\big\},$$
which is closed in $C^1(\overline{\Omega})$ endowed with the standard $C^1(\overline{\Omega})$ norm. For every $v\in E$, the function $x\mapsto\delta g(x,v(x),\nabla v(x))$ is continuous in $\overline{\Omega}$. Let then $T(v)$ be the unique solution of
\be\label{defTv}\left\{\baa{rcll}
\Delta T(v)+(\lambda_1+\epsilon)T(v) & \!\!=\!\! & \delta g(x,v(x),\nabla v(x)) & \!\hbox{in $\Omega$},\vspace{3pt}\\
T(v) & \!\!=\!\! & 0 & \!\hbox{on $\partial\Omega$},\eaa\right.
\ee
that is,
$$T(v)=\mathcal{T}_\epsilon(\delta g(\cdot,v,\nabla v)),$$
with the notation~\eqref{Teps}. The function $T(v)$ belongs to $W^{2,p}(\Omega)\cap W^{1,p}_0(\Omega)$ for all $1\le p<+\infty$, and thus in $C^1(\overline{\Omega})$. Remember that both $\epsilon$ and $\delta$ are negative. From~\eqref{defC}-\eqref{defpsi}, there holds
$$\displaystyle\Delta T(v)+(\lambda_1+\epsilon)T(v)=\delta g(x,v(x),\nabla v(x))\ge\delta C+\frac{\epsilon v(x)}{2}\ge\delta C+\frac{\epsilon\psi(x)}{2},$$
a.e. in $\Omega$, while $\Delta\psi+(\lambda_1+\epsilon)\psi=\delta C+\epsilon\psi/2$ in $\Omega$,
whence $T(v)\le\psi$ in $\overline{\Omega}$ from the maximum principle~\cite{bnv}. Similarly, from~\eqref{hypg},~\eqref{defs0} and~\eqref{defsigma}, there holds
$$\displaystyle\Delta T(v)+(\lambda_1+\epsilon)T(v)=\delta g(x,v(x),\nabla v(x))\le\delta g(x,\sigma\varphi_1(x),\nabla v(x))\le\epsilon\sigma\varphi_1(x)$$
a.e. in $\Omega$, while $\Delta(\sigma\varphi_1)+(\lambda_1+\epsilon)(\sigma\varphi_1)=\epsilon\sigma\varphi_1$ in $\Omega$, whence $T(v)\ge\sigma\varphi_1$ in $\overline{\Omega}$ from the maximum principle~\cite{bnv}.

To sum up, $T$ maps the non-empty closed convex set $E$ into itself. Furthermore, since
$$|\delta g(\cdot,v,\nabla v)|\le|\delta|C+\frac{|\epsilon|\psi}{2}\ \hbox{ in }\Omega$$
for all $v\in E$, and since $\psi$ is positive and bounded, it follows that $T(E)$ is bounded in~$W^{2,p}(\Omega)$ for every $1\le p<+\infty$, and in particular $\overline{T(E)}$ is a compact subset of $C^1(\overline{\Omega})$. Finally, the Schauder fixed point theorem yields the existence of a fixed point of~$T$ in~$E$, that is, there is $u\in E$ (whence, $u>0$ in $\Omega$) belonging to $W^{2,p}(\Omega)\cap W^{1,p}_0(\Omega)$ for all $1\le p<+\infty$, solution to~\eqref{eq} (and then $u$ is a classical $C^2(\Omega)\cap C(\overline{\Omega})$ solution, from the general considerations of Section~\ref{intro}). 

\vskip 0.2cm
\noindent{\it Second case:}
$$\epsilon>0\ \hbox{ and }\ \delta>0.$$
In that case, we are also given a real number $A\ge1$, we assume that $A^{-1}\epsilon\le\delta\le A\epsilon$, and we shall prove the existence of solutions to~\eqref{eq} when $\epsilon>0$ is small enough. First of all, without loss of generality, we consider the principal eigenfunction $\varphi_1$ of~\eqref{varphi1} normalized with $\|\varphi_1\|_{L^2(\Omega)}=1$.

From~\eqref{hypg}, there are some positive constant $C'>0$ and $s'_0>0$ such that
\be\label{defC'}
\forall\,(x,s,p)\in\overline{\Omega}\times[0,+\infty)\times\R^N,\ \ 0\le g(x,s,p)\le C'+\frac{s}{4A\|\varphi_1\|_{L^1(\Omega)}\|\varphi_1\|_{L^\infty(\Omega)}}
\ee
and
\be\label{defs'0}
\forall\,(x,s,p)\in\overline{\Omega}\times[0,s'_0]\times\R^N,\ \ g(x,s,p)\ge2As.
\ee
Fix $\beta>0$ large enough so that
\be\label{defbeta}
A\,C'\|\varphi_1\|_{L^1(\Omega)}+\frac{\beta}{4\|\varphi_1\|_{L^\infty(\Omega)}}\le\frac{\beta}{2\|\varphi_1\|_{L^\infty(\Omega)}},
\ee
and then $\varsigma>0$ small enough so that 
\be\label{defvarsigma}
\varsigma\|\varphi_1\|_{L^\infty(\Omega)}\le\min\Big(s'_0,\frac\beta4\Big).
\ee
Pick any $p\in(N,+\infty)$, and set
\be\label{defM2}
M:=\Big(C'+\frac{\beta}{4A\|\varphi_1\|_{L^1(\Omega)}\|\varphi_1\|_{L^\infty(\Omega)}}\Big)\times|\Omega|^{1/p}>0,
\ee
where $|\Omega|$ denotes the $N$-dimensional Lebesgue measure of $\Omega$. With the parameters $p$, $M$, and
$$0<a:=2\varsigma\le b:=\frac{\beta}{2\|\varphi_1\|_{L^\infty(\Omega)}}<+\infty,$$
Proposition~\ref{pro2} yields the existence of $\epsilon_0>0$ and $\eta_0>0$ such that, for every $\epsilon\in(0,\epsilon_0)$ and for every $w\in L^p(\Omega)\cap E_{\epsilon a,\epsilon b}$ such that $\|w\|_{L^p(\Omega)}\le M$ and $\|w\|_{L^1(\Omega)}\le\eta_0$, then
\be\label{encadrements}
\varsigma\varphi_1\le\mathcal{T}_\epsilon(w)\le\frac{\beta}{\|\varphi_1\|_{L^\infty(\Omega)}}\,\varphi_1\le\beta\ \hbox{ in }\overline{\Omega}.
\ee
Without loss of generality, one can assume that
\be\label{defepsilon0}
0<\epsilon_0\le\min\Big(\lambda_2-\lambda_1,\frac{1}{A}\Big)\ \hbox{ and }\ A\epsilon_0\Big(C'+\frac{\beta}{4A\|\varphi_1\|_{L^1(\Omega)}\|\varphi_1\|_{L^\infty(\Omega)}}\Big)\times|\Omega|\le\eta_0.
\ee

Fix in the sequel any $\epsilon\in(0,\epsilon_0)$ and $\delta$ such that
$$A^{-1}\epsilon\le\delta\le A\epsilon$$
(thus, $0<\delta\le A\epsilon_0\le1$). Consider the non-empty convex set 
$$E':=\big\{v\in C^1(\overline{\Omega}): \varsigma\varphi_1\le v\le\beta\hbox{ in }\overline{\Omega}\big\},$$
which is closed in $C^1(\overline{\Omega})$ endowed with the standard $C^1(\overline{\Omega})$ norm, and consider any $v\in E'$. The function
$$x\mapsto w(x):=\delta g(x,v(x),\nabla v(x))$$
is continuous in $\overline{\Omega}$, and we still call $T(v):=\mathcal{T}_\epsilon(w)$ the unique solution of~\eqref{defTv}. The function $T(v)$ belongs to $W^{2,q}(\Omega)\cap W^{1,q}_0(\Omega)$ for all $1\le q<+\infty$, and thus in $C^1(\overline{\Omega})$. Let us check that the function $w$ fulfills the conditions of the previous paragraph. First of all, $w\in L^p(\Omega)$ and, by~\eqref{hypg} and~\eqref{defC'},
\be\label{ineqw3}
0\le w\le\delta\Big(C'+\frac{\beta}{4A\|\varphi_1\|_{L^1(\Omega)}\|\varphi_1\|_{L^\infty(\Omega)}}\Big)\le C'+\frac\beta{4A\|\varphi_1\|_{L^1(\Omega)}\|\varphi_1\|_{L^\infty(\Omega)}}\ \hbox{ in }\Omega,
\ee
whence $\|w\|_{L^p(\Omega)}\le M$ by~\eqref{defM2}. Furthermore, from~\eqref{defbeta} and the previous line,
$$\int_\Omega w\varphi_1\le\delta\Big(C'\|\varphi_1\|_{L^1(\Omega)}+\frac\beta{4A\|\varphi_1\|_{L^\infty(\Omega)}}\Big)\le\epsilon\Big(AC'\|\varphi_1\|_{L^1(\Omega)}+\frac\beta{4\|\varphi_1\|_{L^\infty(\Omega)}}\Big)\le\frac{\beta\epsilon}{2\|\varphi_1\|_{L^\infty(\Omega)}},$$
while
$$\int_\Omega w\varphi_1\ge\delta\int_\Omega g(x,\varsigma\varphi_1(x),\nabla v(x))\,\varphi_1(x)\,dx\ge2A\delta\varsigma\int_\Omega\varphi_1^2=2A\delta\varsigma\ge2\epsilon\varsigma$$
by~\eqref{hypg},~\eqref{defs'0} and~\eqref{defvarsigma}. Therefore, $w\in E_{\epsilon a,\epsilon b}$. Lastly,~\eqref{defepsilon0} and~\eqref{ineqw3}, together with $\delta\le A\epsilon_0$, imply that
$$\|w\|_{L^1(\Omega)}\le\delta\Big(C'+\frac{\beta}{4A\|\varphi_1\|_{L^1(\Omega)}\|\varphi_1\|_{L^\infty(\Omega)}}\Big)\times|\Omega|\le\eta_0.$$
One then infers from~\eqref{encadrements} that
$$\varsigma\varphi_1\le T(v)=\mathcal{T}_\epsilon(w)\le\beta\ \hbox{ in $\overline{\Omega}$}.$$

To sum up, $T$ maps the non-empty closed convex set $E'$ into itself. Furthermore, since, by~\eqref{ineqw3},
$$0\le w=\delta g(\cdot,v,\nabla v)\le C'+\frac{\beta}{4A\|\varphi_1\|_{L^1(\Omega)}\|\varphi_1\|_{L^\infty(\Omega)}}\ \hbox{ in $\Omega$}$$
for all $v\in E'$, the image $T(E')$ is bounded in $W^{2,q}(\Omega)$ for all $1\le q<+\infty$, and in particular $\overline{T(E')}$ is a compact subset of $C^1(\overline{\Omega})$. Finally, the Schauder fixed point theorem yields the existence of a fixed point of $T$ in $E'$, that is, there is $u\in E'$ (whence, $u>0$ in~$\Omega$) belonging to $W^{2,q}(\Omega)\cap W^{1,q}_0(\Omega)$ for all $1\le q<+\infty$, solution to~\eqref{eq} (and then $u$ is a classical $C^2(\Omega)\cap C(\overline{\Omega})$ solution, from the general considerations of Section~\ref{intro}). The proof of part~$(i)$ of Theorem~\ref{th1} is thereby complete.\hfill$\Box$

\begin{remark}
In the case $\epsilon>0$ and $\delta>0$, even if it meant increasing the constant $\beta>0$, one could have assumed by~\eqref{hypg} without loss of generality that $(\lambda_1+\epsilon)\beta\ge\lambda_1\beta\ge\delta g(\cdot,\beta,0)$ in $\Omega$, that is, the constant $\beta$ would be a sub-solution of~\eqref{eq}. Furthermore, even it is meant decreasing $\varsigma>0$, one could have assumed by~\eqref{hypg} without loss of generality that
$$\Delta(\varsigma\varphi_1)+(\lambda_1+\epsilon)\varsigma\varphi_1=\epsilon\varsigma\varphi_1\le\delta g(\cdot,\varsigma\varphi_1,\varsigma\nabla\varphi_1)$$
in $\Omega$, that is, the function $\varsigma\varphi_1$ would be a super-solution of~\eqref{eq}. The previous proof then shows the existence of a solution $u$ to~\eqref{eq} between the super-solution $\varsigma\varphi_1$ and the sub-solution $\beta$, ordered in the unsual way.
\end{remark}


\section{Proof of Theorem~\ref{th1}, parts~$(ii)$-$(iii)$: convergence and log-concavity}\label{sec4}

{\it Proof of part~$(ii)$ of Theorem~$\ref{th1}$.} We start with considering any sequence $(\epsilon_n,\delta_n)_{n\in\N}$ converging to $(0^+,0^+)$ or $(0^-,0^-)$ and for which there is $A\ge1$ such that
\be\label{A2}
\forall\,n\in\N,\ \ A^{-1}|\epsilon_n|\le|\delta_n|\le A\,|\epsilon_n|.
\ee
We then let $(u_n)_{n\in\N}$ be a sequence of solutions to~\eqref{eq}, with parameters $(\epsilon_n,\delta_n)$ instead of $(\epsilon,\delta)$. Such solutions exist for all $n$ large enough, from part~$(i)$ of Theorem~\ref{th1}. From the considerations of Section~\ref{intro}, these solutions $u_n$ belong to $W^{2,q}(\Omega)\cap W^{1,q}_0(\Omega)$ for all $1\le q<+\infty$, and then to $C^{1,\gamma}(\overline{\Omega})$ for all $0<\gamma<1$. They are also of class $C^2(\Omega)$.

Let us first show that the sequence $(u_n)_{n\in\N}$ is bounded in $W^{2,q}(\Omega)$, for every $1\le q<+\infty$. From~\eqref{hypg}, there exists a constant $C>0$ such that
\be\label{defC3}
\forall\,(x,s,p)\in\overline{\Omega}\times[0,+\infty)\times\R^N,\ \ 0\le g(x,s,p)\le C+\frac{s}{2A}.
\ee
By integrating~\eqref{eq} against $\varphi_1$, we get that
\be\label{integrals2}
\epsilon_n\int_\Omega u_n(x)\,\varphi_1(x)\,dx=\delta_n\int_\Omega g(x,u_n(x),\nabla u_n(x))\,\varphi_1(x)\,dx,
\ee
whence
$$\baa{rcl}
\displaystyle|\epsilon_n|\int_\Omega u_n(x)\,\varphi_1(x)\,dx & \le & \displaystyle|\delta_n|\int_\Omega\Big(C+\frac{u_n(x)}{2A}\Big)\varphi_1(x)\,dx\vspace{3pt}\\
& \le & \displaystyle A\,C\,|\epsilon_n|\,\|\varphi_1\|_{L^1(\Omega)}+\frac{|\epsilon_n|}{2}\int_\Omega u_n(x)\,\varphi_1(x)\,dx.\eaa$$
As a consequence, the sequence
$$(\theta_n)_{n\in\N}:=\Big(\frac{1}{\|\varphi_1\|_{L^2(\Omega)}^2}\int_\Omega u_n\varphi_1\Big)_{n\in\N}$$
is bounded. By rewriting~\eqref{eq} as
$$\left\{\baa{rcll}
\Delta(u_n-\theta_n\varphi_1)+(\lambda_1+\epsilon_n)(u_n-\theta_n\varphi_1) & = & \delta_n\,g(x,u_n,\nabla u_n)-\epsilon_n\theta_n\varphi_1 & \hbox{in }\Omega,\vspace{3pt}\\
u_n-\theta_n\varphi_1 & = & 0 & \hbox{on }\partial\Omega,\eaa\right.$$
by multiplying the above equation by $u_n-\theta_n\varphi_1$ (belonging to $W^{2,q}(\Omega)\cap W^{1,q}_0(\Omega)\cap C^1(\overline{\Omega})$ for each $1\le q<+\infty$), and by integrating by parts, one gets that
$$\baa{rcl}
\displaystyle\int_\Omega|\nabla(u_n\!-\!\theta_n\varphi_1)|^2\!-\!(\lambda_1\!+\!\epsilon_n)\int_\Omega(u_n\!-\!\theta_n\varphi_1)^2 & \!\!=\!\! & \displaystyle-\delta_n\int_\Omega g(x,u_n,\nabla u_n)\,(u_n-\theta_n\varphi_1)\vspace{3pt}\\
& \!\!\!\! & \displaystyle+\epsilon_n\theta_n\int_\Omega(u_n-\theta_n\varphi_1)\,\varphi_1\vspace{3pt}\\
& \!\!\le\!\! & \displaystyle|\delta_n|\int_\Omega\Big(C+\frac{u_n}{2A}\Big)\,|u_n-\theta_n\varphi_1|\vspace{3pt}\\
& \!\!\!\! & + |\epsilon_n|\,|\theta_n|\,\|u_n\!-\!\theta_n\varphi_1\|_{L^2(\Omega)}\,\|\varphi_1\|_{L^2(\Omega)}.\eaa$$
On the other hand, the inequality~\eqref{vn1} still holds, with $v_n$ replaced by $u_n$. Hence, by assuming without loss of generality that $\epsilon_n<(\lambda_2-\lambda_1)/2$, one gets from the Cauchy-Schwarz inequality that
$$\baa{l}
\displaystyle\frac{\lambda_2-\lambda_1}{2}\|u_n-\theta_n\varphi_1\|_{L^2(\Omega)}^2\vspace{3pt}\\
\qquad\qquad\le\displaystyle\Big(|\delta_n|\,C\,|\Omega|^{1/2}+\frac{|\delta_n|\,|\theta_n|\,\|\varphi_1\|_{L^2(\Omega)}}{2A}+|\epsilon_n|\,|\theta_n|\,\|\varphi_1\|_{L^2(\Omega)}\Big)\,\|u_n-\theta_n\varphi_1\|_{L^2(\Omega)}\vspace{3pt}\\
\qquad\qquad\ \ \ \ \displaystyle+\frac{|\delta_n|}{2A}\|u_n-\theta_n\varphi_1\|_{L^2(\Omega)}^2.\eaa$$
Since $(|\epsilon_n|,|\delta_n|)\to(0^+,0^+)$ as $n\to+\infty$ and the sequence $(\theta_n)_{n\in\N}$ is bounded, it follows that the sequence $(u_n-\theta_n\varphi_1)_{n\in\N}$ is bounded in $L^2(\Omega)$ and even converges to $0$ in $L^2(\Omega)$. From the boundedness of $(\theta_n)_{n\in\N}$ again, there is then a real number $B$ such that, up to extraction of a subsequence,
$$u_n\to B\varphi_1\ \hbox{ in $L^2(\Omega)$ as $n\to+\infty$}.$$
Since each function $u_n$ is positive in $\Omega$, and $\varphi_1$ is positive too, there holds $B\ge0$. However, we want to have the convergence in a stronger sense, and also to show that $B$ is positive.

By rewriting~\eqref{eq} (with $\epsilon_n$ and $\delta_n$) as
\be\label{eqter}\left\{\baa{rcll}
\Delta u_n & = & \delta_n\,g(x,u_n,\nabla u_n)-(\lambda_1+\epsilon_n)\,u_n & \hbox{in }\Omega,\vspace{3pt}\\
u_n & = & 0 & \hbox{on }\partial\Omega,\eaa\right.
\ee
and by using~\eqref{defC3} together with the boundedness of the sequences $(\epsilon_n)_{n\in\N}$, $(\delta_n)_{n\in\N}$ and $(\|u_n\|_{L^2(\Omega)})_{n\in\N}$, the sequence of right-hand sides of the above equation is bounded in $L^2(\Omega)$, whence $(u_n)_{n\in\N}$ is bounded in $W^{2,2}(\Omega)$. By a bootstrap procedure, it follows finally that the sequence $(u_n)_{n\in\N}$ belongs to $W^{2,q}(\Omega)\cap W^{1,q}_0(\Omega)$ and is bounded in $W^{2,q}(\Omega)$ for every $1\le q<+\infty$. Therefore,
\be\label{limitsB}
u_n\to B\varphi_1\ \hbox{ in $C^1(\overline{\Omega})$ as $n\to+\infty$}
\ee
(and even $u_n\to B\varphi_1$ in $C^{1,\gamma}(\overline{\Omega})$ as $n\to+\infty$ for every $0<\gamma<1$). Let us now show that~$B>0$. Assume by way of contradiction that $B=0$, and let $s'_0>0$ be as in~\eqref{defs'0}. Then $0<u_n\le s'_0$ in $\Omega$ for all $n$ large enough, and~\eqref{defs'0} and~\eqref{integrals2} entail that
$$\frac{\epsilon_n}{\delta_n}\int_\Omega u_n(x)\,\varphi_1(x)\,dx=\int_\Omega g(x,u_n(x),\nabla u_n(x))\,\varphi_1(x)\,dx\ge 2A\int_\Omega u_n(x)\,\varphi_1(x)\,dx,$$
contradicting~\eqref{A2}. Therefore,
$$B>0.$$

Assuming now that $\delta_n\sim c\,\epsilon_n$ as $n\to+\infty$ for some $c>0$, one derives from~\eqref{integrals2} and the continuity of $g$ in $\overline{\Omega}\times[0,+\infty)\times\R^N$ that
$$B\,\|\varphi_1\|_{L^2(\Omega)}^2=c\int_\Omega g(x,B\varphi_1(x),B\nabla\varphi_1(x))\,\varphi_1(x)\,dx.$$
Therefore, if one further assumes that $g(x,s,p)=g(x,s)$ is independent of $p$ and satisfies~\eqref{hypg2}, one has
$$\|\varphi_1\|_{L^2(\Omega)}^2=c\int_\Omega\frac{g(x,B\varphi_1(x))}{B}\,\varphi_1(x)\,dx,$$
and the quantity in the right-hand is decreasing with respect to $B>0$ if considered as a parameter. As a consequence, $B>0$ in~\eqref{limitsB} is uniquely determined by $\Omega$, $g$, $c$ and $\varphi_1$, and does not depend on any sequence $(\epsilon_n,\delta_n)_{n\in\N}$ converging to $(0^+,0^+)$ or $(0^-,0^-)$ with $\delta_n\sim c\,\epsilon_n$ as $n\to+\infty$. This actually shows part~$(ii)$ of Theorem~\ref{th1}.

\vskip 0.3cm
\noindent{\it Proof of part~$(iii)$ of Theorem~$\ref{th1}$.} Here, $\partial\Omega$ is of class $C^{2,\alpha}$ for some $\alpha\in(0,1)$, and $\Omega$ is strictly convex. We fix $A\ge1$, and assume by way of contradiction that the conclusion does not hold. Then there is a sequence $(\epsilon_n,\delta_n)_{n\in\N}$ converging to $(0^+,0^+)$ or $(0^-,0^-)$ and satisfying~\eqref{A2}, together with a sequence of solutions $(u_n)_{n\in\N}$ of~\eqref{eq} with $(\epsilon,\delta)$ replaced by $(\epsilon_n,\delta_n)$, such that $u_n$ is not log-concave in $\Omega$, for every $n\in\N$. Therefore, there are a sequence of points $(x_n)_{n\in\N}$ in $\Omega$ and a sequence of unit vectors $(\xi_n)_{n\in\N}$ in $\R^N$ such that
\be\label{Hnxn}
\xi_n\,H_n(x_n)\,\xi_n>0\ \hbox{ for all $n\in\N$},
\ee
where $H_n(x_n):=(D^2\log u_n)(x_n)$ denotes the Hessian matrix of $\log u_n$ at $x_n$, and $\xi H\xi:=\sum_{1\le i,j\le N}H_{i,j}\xi_i\xi_j$ for any $N\times N$ symmetric matrix $H$ and any $\xi\in\R^N$.

From the previous paragraphs, there is $B>0$ such that, up to extraction of a subsequence,
\be\label{unB}
u_n\to B\varphi_1\hbox{ as $n\to+\infty$}
\ee
in $C^{1,\gamma}(\overline{\Omega})$ for every $0<\gamma<1$. From the local H\"older-continuity of $g$ in~\eqref{hypg}, there is then $\alpha'>0$ such that the sequence $(g(\cdot,u_n,\nabla u_n))_{n\in\N}$ is bounded in $C^{0,\alpha'}(\overline{\Omega})$. Using~\eqref{eqter} together with the boundedness of $(\epsilon_n,\delta_n)_{n\in\N}$ and the $C^{2,\alpha}$ smoothness of $\partial\Omega$, one infers that the sequence $(u_n)_{n\in\N}$ is bounded in $C^{0,\min(\alpha,\alpha')}(\overline{\Omega})$. Thus, the convergence~\eqref{unB} actually holds in $C^2(\overline{\Omega})$.

Up to extraction of another subsequence, there are $x\in\overline{\Omega}$ and a unit vector $\xi$ in $\R^N$ such that $x_n\to x$ and $\xi_n\to\xi$ as $n\to+\infty$. Two cases may occur, according to the location of $x$. If $x\in\Omega$, then
$$\xi_n\,H_n(x_n)\,\xi_n\to\frac{\varphi_1(x)\partial_{\xi\xi}\varphi_1(x)-(\partial_\xi\varphi_1(x))^2}{(\varphi_1(x))^2}=\xi\,(D^2\log\varphi_1)(x)\,\xi\ \hbox{ as $n\to+\infty$}.$$
But $\xi\,(D^2\log\varphi_1)(x)\,\xi<0$ from~\cite[Lemma~2.5]{lv}, and we then get a contradiction with~\eqref{Hnxn}. Therefore, $x\in\partial\Omega$. Two cases may then occur, according to the direction of $\xi$. If $\xi\cdot\nu(x)\neq0$, where $\nu(x)$ denotes the outward unit normal to $\Omega$ at $x$, then $\partial_\xi\varphi_1(x)\neq0$ from the Hopf lemma (more precisely, $\partial_\xi\varphi_1(x)>0$ if $\xi\cdot\nu(x)<0$, and $\partial_\xi\varphi_1(x)=-\partial_{-\xi}\varphi_1(x)<0$ if $\xi\cdot\nu(x)>0$), whence
$$\xi_n\,H_n(x_n)\,\xi_n=\frac{u_n(x_n)\partial_{\xi_n\xi_n}u_n(x_n)-(\partial_{\xi_n}u_n(x_n))^2}{(u_n(x_n))^2}\sim\frac{-(\partial_{\xi_n}u_n(x_n))^2}{(u_n(x_n))^2}\to-\infty\ \hbox{ as $n\to+\infty$}$$
since $0<u_n(x_n)\to\varphi_1(x)=0$, $\partial_{\xi_n\xi_n}u_n(x_n)\to\partial_{\xi\xi}\varphi_1(x)$ and $\partial_{\xi_n}u_n(x_n)\to\partial_\xi\varphi_1(x)\neq0$ as $n\to+\infty$. This again contradicts~\eqref{Hnxn}. Therefore, $\xi\cdot\nu(x)=0$. In that case, from the strict convexity of $\Omega$, one knows that $\partial_{\xi\xi}\varphi_1(x)<0$, see the proof of~\cite[Lemma~2.5]{lv}. Therefore,
$$\baa{rcl}
\displaystyle\limsup_{n\to+\infty}\,\xi_n\,H_n(x_n)\,\xi_n & = & \displaystyle\limsup_{n\to+\infty}\frac{u_n(x_n)\partial_{\xi_n\xi_n}u_n(x_n)-(\partial_{\xi_n}u_n(x_n))^2}{(u_n(x_n))^2}\vspace{3pt}\\
& \le & \displaystyle\limsup_{n\to+\infty}\frac{\partial_{\xi_n\xi_n}u_n(x_n)}{u_n(x_n)}=-\infty\eaa$$
since $\partial_{\xi_n\xi_n}u_n(x_n)\to\partial_{\xi\xi}\varphi_1(x)<0$ and $0<u_n(x_n)\to\varphi_1(x)=0$ as $n\to+\infty$. This again contradicts~\eqref{Hnxn}.

As a conclusion, all possible cases regarding the limits $x$ and $\xi$ are ruled out. That shows that the assumption~\eqref{Hnxn} was impossible. The proof of part~$(iii)$ of Theorem~\ref{th1} is thereby complete.\hfill$\Box$

\begin{remark}
The above arguments actually show that the solutions $u_n$ are not only log-concave, but also uniformly strictly log-concave for all $n$ large enough, in the sense that there is $\rho>0$, independent of the sequences $(\epsilon_n,\delta_n)_{n\in\N}$ and $(u_n)_{n\in\N}$, such that $D^2\log u_n\le-\rho I_N$ for all $n$ large enough in the sense of symmetric matrices, where $I_N$ denotes the $N\times N$ identity matrix.
\end{remark}


{\footnotesize{\bibliographystyle{plain}
\bibliography{biblio}}}

\end{document}